\RequirePackage[l2tabu, orthodox]{nag}

\documentclass[reqno]{amsart}

\usepackage[fontsize=11pt]{scrextend}

\usepackage{lmodern}
\usepackage[T1]{fontenc}
\usepackage[utf8]{inputenc}
\usepackage[english]{babel}
\usepackage[margin=2.8cm]{geometry}

\usepackage{microtype} 
\usepackage{comment}
\usepackage{varwidth} 

\usepackage{amsmath,amssymb,amsthm,mathrsfs,listings,xskak,
latexsym,mathtools,mathdots,booktabs,enumitem,bm,url}
\usepackage[centertableaux]{ytableau}
\usepackage[pdftex]{hyperref}
\usepackage[capitalize,noabbrev]{cleveref}
\usepackage{caption}
\usepackage{subcaption}

\usepackage{todonotes}
\presetkeys%
    {todonotes}%
    {inline,backgroundcolor=yellow}{}

\newtheorem{theorem}{Theorem}[section]
\newtheorem{proposition}[theorem]{Proposition}
\newtheorem{lemma}[theorem]{Lemma}
\newtheorem{corollary}[theorem]{Corollary}

\theoremstyle{definition}

\newtheorem{example}[theorem]{Example}

\definecolor{lightblue}{rgb}{0.8,0.8,1.0}
\definecolor{lightgreen}{rgb}{0.8,1.0,0.8}
\definecolor{pBlue}{RGB}{86,139,190}
\definecolor{pCyan}{RGB}{149,186,201}
\definecolor{pSand}{RGB}{184,166,121}
\definecolor{pAlgae}{RGB}{87,115,135}
\definecolor{pSkin}{RGB}{236,216,167}
\definecolor{pGray}{RGB}{156,175,156}
\definecolor{pPink}{RGB}{215,114,127}
\definecolor{pOrange}{RGB}{211,153,80}

  \tikzset{
dot/.style = {circle, fill, minimum size=#1,
              inner sep=0pt, outer sep=0pt},
dot/.default = 6pt 
}

\newcommand{\defin}[1]{%
\relax\ifmmode%
\textcolor{blue}{#1}%
\else\textcolor{blue}{\emph{#1}}%
\fi%
}

\newcommand{\setZ}{\mathbb{Z}}

\newcommand{\setR}{\mathbb{R}}

\newcommand{\NN}{\text{NN}}

\tikzset{every picture/.append
	style={
		scale=1,
		x=1em,
		y=1em,
		entries/.style={xshift=-0.5em,yshift=-0.5em,font=\small},
		thickLine/.style={line width=1.4pt,line join=round},
		bgEntry/.style={xshift=-0.5em,yshift=-0.5em,
			regular polygon,regular polygon sides=4,fill,inner sep=0pt,minimum size=1.35em
		}
	}
}
\usetikzlibrary{calc}
\DeclareMathOperator{\des}{des}

\newcommand{\precdot}{\prec\mathrel{\mkern-5mu}\mathrel{\cdot}}

\makeatletter
\newcommand{\preceqdot}{\mathrel{\mathpalette\pr@ceqd@t\relax}}
\newcommand{\pr@ceqd@t}[2]{%
  \begingroup
  \sbox\z@{$#1\prec$}\sbox\tw@{$#1\preceq$}%
  \dimen@=\dimexpr\ht\tw@-\ht\z@\relax
  {\preceq}%
  \mkern-5mu
  \raisebox{\dimen@}{$\m@th#1\cdot$}%
  \endgroup
}
\makeatother

\newcommand{\interl}{\ll}

\title[Order polytopes]{Order polytopes of generalized snake posets are $h^*$-real-rooted}

\author[Braun]{Benjamin Braun}
\address{Department of Mathematics\\
         University of Kentucky
}
\email{benjamin.braun@uky.edu}
\urladdr{https://sites.google.com/view/braunmath/}

\author[Jal]{Aryaman Jal}
\address{Fakult\"{a}t f\"{u}r Mathematik, Ruhr-Universit\"{a}t Bochum, Bochum, Germany}

\email{aryaman.jal@rub.de}

\begin{document}

\begin{abstract}
Order polytopes for generalized snake posets were recently studied by von Bell et al. (2022), and are known to be unimodularly equivalent to strength-one flow polytopes for acyclic directed graphs strongly dual to generalized snake posets.
Lee, Vindas-Mel\'endez, and Wang (2026) conjectured that the Ehrhart $h^*$-polynomials of these order polytopes are real-rooted.
We prove this conjecture using a connection between these $h^*$-polynomials and non-nesting rook polynomials, which were recently introduced by Alexandersson and Jal (2024+) in connection with $P$-Eulerian polynomials for width two posets.
\end{abstract}

\thanks{BB was partially supported by NSF award DMS-2450299. AJ was supported by the SPP 2458 ``Combinatorial Synergies'', funded by the Deutsche Forschungsgemeinschaft (DFG, German Research Foundation) – project number 539866596. The authors thank the Max Planck Institute for Mathematics in the Sciences for hosting the conference ``Recent Developments on Lattice Polytopes in Leipzig'', where this project was initiated.}

\maketitle

\section{Introduction}\label{sec:intro}

Order polytopes for generalized snake posets, which are width two distributive lattices, were studied by von Bell et al.~\cite{vonBellBraunetal2022GenSnakePosets}, with a focus on inequalities satisfied by their normalized volumes.
These inequalities were generalized by Lee, Vindas-Mel\'endez, and Wang~\cite{lee2024generalizedsnakeposetsorder} to coefficient-wise inequalities for the Ehrhart $h^*$-polynomials of these order polytopes, and they conjectured that these $h^*$-polynomials are real-rooted.
Our goal in this work is to prove this conjecture.

Generalized snake posets are strongly planar.
Hence, a result of M\'esz\'aros, Morales, and Striker~\cite{MMS} implies that for a generalized snake poset $P$, the order polytope of $P$ is unimodularly equivalent to the strength-one flow polytope of the acyclic directed graph (DAG) strongly dual to $P$.
The family of flow polytopes arising via strong duals of generalized snake posets has recently arisen in the classification of flow polytopes for amply framed DAGs that admit a projection to a locally antiblocking $\mathbf{g}$-polytope~\cite{berggren2026locallyantiblockingmathbfgpolytopesflow}.
Our results imply that the flow polytopes for these dual DAGs are also $h^*$-real-rooted.
It is unknown whether or not all flow polytopes have real-rooted $h^*$-polynomials, though some families of flow polytopes are known to have this property~\cite{BBDH}.

The $h^*$-polynomial of an order polytope is the $P$-Eulerian polynomial of the corresponding poset with a natural labeling, and many important polynomials in enumerative combinatorics arise in this context. 
For example, the $P$-Eulerian polynomials for the special cases of snake and ladder posets are respectively the Delannoy and Narayana polynomials~\cite{lee2024generalizedsnakeposetsorder}, both of which are real-rooted.
Further importance of $P$-Eulerian polynomials in enumerative and algebraic combinatorics arises from the Neggers--Stanley conjecture, which posited the real-rootedness of $W_{P}(t)$~\cite{Neggers1978, Stanley86LogConcavitySurvey}. 
Br\"{a}ndén~\cite{Branden2004NeggersStanley} and Stembridge~\cite{Stembridge2007NeggersStanleyCounter} gave counterexamples to this conjecture in its arbitrarily labeled and naturally labeled variant, respectively. 
The counterexamples came from the class of width two posets, and generalized snake posets fall within this class. 
Studying distributional properties of polynomials with non-negative real coefficients, such as unimodality, log-concavity, etc., is a classical topic in enumerative combinatorics~\cite{Stanley86LogConcavitySurvey, Brenti1989UnimodalAMS}.
Determining further distributional properties of $P$-Eulerian polynomials is an open question, although Alexandersson and Jal~\cite{AlexanderssonJal2024RookMatroid} recently established the ultra-log-concavity of these polynomials in the case of naturally labeled width two posets.  

Our proof of real-rootedness in the generalized snake poset case involves the theory of non-nesting rook placements, recently introduced by Alexandersson and Jal~\cite{AlexanderssonJal2024RookMatroid} in the context of rook matroids and $P$-Eulerian polynomials of naturally labeled width two posets $P$.
When $P$ is of width two, Alexandersson and Jal prove that the $P$-Eulerian polynomial for a naturally labeled poset $P$ is also equal to the non-nesting rook polynomial of a skew shape corresponding to $P$. 
This connection also yielded the ultra-log-concavity of the sequence of coefficients of the $P$-Eulerian polynomial of such posets. 
In the present work, the interpretation of $h^*$-polynomials as non-nesting rook polynomials reformulates, and yields simpler proofs of, the recursions originally found by Lee, Vindas-Mel\'endez, and Wang. 
Standard techniques from the theory of interlacing polynomials then allow us to establish real-rootedness for the $P$-Eulerian polynomials of generalized snake posets.

The remainder of our paper is structured as follows.
In Section~\ref{sec:preliminaries}, we review necessary background and establish several preliminary results.
In Section~\ref{sec:polynomials}, we prove several technical results regarding Narayana polynomials and we study the structure of $h^*$-polynomials for order polytopes of generalized snake posets using non-nesting rook placements.
In Section~\ref{sec:realrootedness}, we prove our main result that order polytopes of generalized snake posets are $h^*$-real-rooted.
In Section~\ref{sec:aidisclosure}, we specify how LLM tools were used in this work.

\section{Preliminaries}\label{sec:preliminaries}

We use $[n]$ to mean the set $\{1,2, \dotsc , n\}$ and $[m, n]$ to mean the 
discrete interval $\{m, m+1, \dotsc , n\}$. 
The set of size-$k$ subsets of $[n]$ is denoted by $\binom{[n]}{k}$. 

\subsection{Posets, Ehrhart theory, and $P$-Eulerian polynomials}

Given a poset $P$ on the elements $[n]$ with partial order $\prec$, the \defin{order polytope} of $P$ is the polytope
\[
\mathcal{O}(P):= \{x\in \setR^n : 0\leq x_i\leq 1 \text{ for }i=1,\ldots,n \text{ and } x_i\leq x_j \text{ when } i\prec j \}\, . 
\]
The order polytope is known to be a lattice polytope with vertices given by indicator vectors of upper order ideals of $P$~\cite{Stanley86TwoPosetPolytopes}.
Given a $d$-dimensional lattice polytope $\mathcal{P}\subset \setR^n$, the \defin{Ehrhart series} of $\mathcal{P}$ is the rational generating function
\[
\sum_{k\geq 0}|k\mathcal{P}\cap \setZ^n|t^k = \frac{\sum_{j=0}^dh_j^*t^j}{(1-t)^{d+1}} \, .
\]
Rationality of this series is due to Ehrhart~\cite{ehrhart}.
The polynomial $h^*(\mathcal{P};t):=\sum_{j=0}^dh_j^*t^j$ is called the \defin{$h^*$-polynomial} of $\mathcal{P}$, and a theorem of Stanley~\cite{stanleydecompositions} asserts that $h^*_j\in \setZ_{\geq 0}$ for all $j$.

For the order polytope of a finite poset, Stanley~\cite{Stanley86TwoPosetPolytopes} gave a combinatorial interpretation for the coefficients of the $h^*$-polynomial that we now recall. A labeling of a poset $P$ on $n$ elements is a bijection $\omega: P \to [n]$. We say that $\omega$ is a \defin{natural labeling}, or linear extension of $P$, if $i \prec j$ implies $\omega(i) < \omega(j).$  The \defin{Jordan--H\"{o}lder set} of $(P, \omega)$ is the set of all permutations of $[n]$ the inverses of which are linear extensions of $(P, \omega)$: 
\[
\mathcal{L}(P, \omega) := \{\sigma \in \mathfrak{S}_{n}: i \prec j \implies \sigma^{-1}({\omega(i)}) < \sigma^{-1}({\omega(j)})\}.
\]
In other words, $\sigma \in \mathcal{L}(P, \omega)$ if for every relation $i \prec j$ in $P$, we have that $\omega(i)$ precedes $\omega(j)$ in the one-line notation of the permutation $\sigma.$
The \defin{$(P, \omega)$-Eulerian polynomial} is the descent-generating polynomial of the Jordan--H\"{o}lder set of $(P, \omega)$: \[
W_{P, \omega}(t) \coloneqq \sum_{\sigma \in \mathcal{L}(P, \omega)}t^{\des(\sigma)}.
\]
When $\omega$ is natural, we simply write $W_{P}$ for this polynomial, since it is independent of the choice of natural labeling.
Stanley proved the following theorem relating $P$-Eulerian polynomials and $h^*$-polynomials.

\begin{theorem}[{Stanley,~\cite{Stanley86TwoPosetPolytopes}}]
    \label{thm:stanleyorderwpoly}
    Given a finite naturally labeled poset $P$, we have
    \[
    h^*(\mathcal{O}(P);t)=W_P(t) \, .
    \]
\end{theorem}

Finally, we define the \defin{width} of $P$ to be the size of the largest antichain in $P$. 

\subsection{Rook placements and generalized snake posets}

An \defin{integer partition} $\lambda = (\lambda_{1}, \ldots , \lambda_{n})$ is a finite sequence of positive integers satisfying $\lambda_{1} \geq \lambda_{2} \geq \ldots \geq \lambda_{n}$. 
We draw the \defin{Ferrers diagram} of $\lambda$ as a left-justified, weakly 
decreasing array of boxes, as in the English convention. 
We denote by $\delta_{n}$ the staircase partition $\delta_{n} = (n, n-1, \ldots , 1)$. 
We write $\mu \subseteq \lambda$ to indicate containment of partitions in the sense of containment of their Ferrers diagrams. 
Given two integer partitions $\lambda \supseteq \mu$ of length at most $r$, we call the 
board with cells
\[
  \left\{ (i, j) : 1 \leq i \leq r \text{ and } \mu_i <  j \leq \lambda_i \right\}
\]
the \defin{skew Ferrers board} associated with $\lambda/\mu$.
We also refer to a skew Ferrers board as a \defin{skew shape}.  
The \defin{size} of a skew shape $\lambda/\mu$, denoted by  $|\lambda/\mu|$, is the number of boxes in its diagram. 
An \defin{outer corner} of $\lambda/\mu$ is a cell $(i, j) \notin \lambda/\mu$ such that the cells to the left and above it are in $\lambda/\mu$. 
An \defin{inner corner} of $\lambda/\mu$ is a cell $(i, j) \notin \lambda/\mu$ such that the cells to the right and below are in $\lambda/\mu$. A skew shape is a \defin{squarecase} if every inner corner and every outer corner has a vertex on the antidiagonal of the shape.

\begin{example}
    \label{ex:outerinnersquarecase}
    Consider the skew shape given on the right-hand side of Figure~\ref{fig:snake_poset_example}.
    As every outer corner and every inner corner lies on the antidiagonal, i.e., the diagonal line from the bottom-left corner to the top-right corner, this is a squarecase shape.
\end{example}

Given a skew board, a non-attacking rook 
placement $\rho$ on a board $B$ is a subset of cells of $B$, called rooks, for which no two cells share a row or column index. 
The size of a rook placement is denoted by $|\rho|$. 
Two rooks in a non-attacking rook placement are \defin{nesting} if 
one rook lies strictly below and to the right of another. 
A \defin{non-nesting rook placement} on a 
board $B$ is a non-attacking rook placement such that no pair of rooks forms a nesting.
Denote the set of non-nesting rook placements on $\lambda/\mu$ by $\NN(\lambda/\mu)$.  
We define the \defin{non-nesting rook polynomial} of $\lambda/\mu$ to be 
\[
M_{\lambda/\mu}(t) \coloneqq \sum_{\rho \in \NN(\lambda/\mu)}t^{|\rho|} \, .
\]
Observe that $\lambda/\mu$, the conjugate shape $\lambda'/\mu'$, and the half-circle rotation   (denoted by $(\lambda/\mu)^{\mathrm{R}}$) all have the same non-nesting rook polynomial. 

The following result, best seen as a consequence of Dilworth's theorem~\cite{Dilworth1950DecompositionTheorem}, has been rediscovered in various forms~\cite{FelsnerWernisch1997MarkovChains, ChanPakPanova2022CrossProdConj, AlexanderssonJal2024RookMatroid}; we follow the conventions of Alexandersson and Jal~\cite{AlexanderssonJal2024RookMatroid}.

\begin{theorem}[{Alexandersson and Jal,~\cite[Theorem 5.4, Lemma 5.6]{AlexanderssonJal2024RookMatroid}}]
\label{thm:correspondence}
    There is a bijective correspondence between (naturally labeled) posets $P$ of width two and skew shapes $\lambda/\mu$. Under this correspondence, the $P$-Eulerian polynomial of $P$ is equal to the non-nesting rook polynomial of $\lambda/\mu$: $W_{P}(t) = M_{\lambda/\mu}(t)$. 
\end{theorem}

This correspondence endows $P$ with a canonical natural labeling and labels the cells of the skew shape $\lambda/\mu$ in accordance with this labeling. 

Generalized snake posets, defined below, are a special class of width two posets.
Through the correspondence given in Theorem~\ref{thm:correspondence}, generalized snake posets correspond to a subclass of squarecases.
Geometric properties of their order polytopes were studied in~\cite{vonBellBraunetal2022GenSnakePosets} and~\cite{lee2024generalizedsnakeposetsorder}.
Given a word $\mathbf{w} = \varepsilon w_{1} \ldots w_{n}$ where $w_{i} \in \{L, R\}$, the \defin{generalized snake poset} $P( \mathbf{w})$ is the width two poset on $2(n+1)$ elements partitioned into two chains $C_{1} = \{1,3,5,\ldots , 2n+1\}$ and $C_{2} =\{2,4,6,\ldots , 2n+2\}$, together with the cover relations \begin{align*}
        2(n+1-i) &\precdot 2(n+1-i) +1 \qquad \text{for all $i \in [n]$ with $w_{i} = R$,} \\  2(n-i)+1 &\precdot 2(n-i+2) \, \, \,\, \quad \qquad \text{for all $i \in [n]$ with $w_{i} = L$.}
\end{align*}
When the word is $\mathbf{w} = \varepsilon LLL \cdots$, we get the \defin{ladder posets} $L_{n}$  and when the word is $\mathbf{w} = \varepsilon LRLR \cdots$, we get the \defin{regular snake posets} $S_{n}$.
Given a word $\mathbf{w}=\varepsilon w_1 w_2\cdots w_n$ defining a generalized snake poset, we denote by $M_{\mathbf{w}}(t)$ the non-nesting rook polynomial for the skew shape corresponding to the poset $P(\mathbf{w})$.
Note that every generalized snake poset is naturally labeled by its elements, and thus we have the equalities
\[
h^*(\mathcal{O}(P(\mathbf{w}));t)=M_{\mathbf{w}}(t)=W_{P(\mathbf{w})}(t) \, .
\]

\begin{example}
The generalized snake poset for $\mathbf{w} = \varepsilon RRLLLRLRR$ and its corresponding skew shape are drawn in Figure~\ref{fig:snake_poset_example}. The associated $P$-Eulerian polynomial is
 \[
M_{\mathbf{w}}(t) = M_{\lambda/\mu}(t) = t^{10} + 24t^{9} + 217t^{8} + 962t^{7} + 2291t^{6} + 3048t^{5} + 2291t^{4} + 962t^{3} + 217t^{2} + 24t + 1\, .
\]
It can be verified computationally that this polynomial is real-rooted.

\begin{figure}
    \centering
    \begin{subfigure}{0.40\textwidth}
    \centering
  \scalebox{0.9}{
\begin{tikzpicture}
\node[dot=4pt, fill=black, label={[xshift=-0.55cm, yshift=-0.25cm, text=blue]{$1$}}] at (0, 0) (a1) {};
\node[dot=4pt, fill=black, label={[xshift=-0.55cm, yshift=-0.25cm, text=blue]{$3$}}] at (0, 2) (a2) {};
\node[dot=4pt, fill=black, label={[xshift=-0.55cm, yshift=-0.25cm, text=blue]{$5$}}] at (0, 4) (a3) {};
\node[dot=4pt, fill=black, label={[xshift=-0.55cm, yshift=-0.25cm, text=blue]{$7$}}] at (0, 6) (a4) {};
\node[dot=4pt, fill=black, label={[xshift=-0.55cm, yshift=-0.25cm, text=blue]{$9$}}] at (0, 8) (a5) {};
\node[dot=4pt, fill=black, label={[xshift=-0.55cm, yshift=-0.25cm, text=blue]{$11$}}] at (0, 10) (a6) {};
\node[dot=4pt, fill=black, label={[xshift=-0.55cm, yshift=-0.25cm, text=blue]{$13$}}] at (0, 12) (a7) {};
\node[dot=4pt, fill=black, label={[xshift=-0.55cm, yshift=-0.25cm, text=blue]{$15$}}] at (0, 14) (a8) {};
\node[dot=4pt, fill=black, label={[xshift=-0.55cm, yshift=-0.25cm, text=blue]{$17$}}] at (0, 16) (a9) {};
\node[dot=4pt, fill=black, label={[xshift=-0.55cm, yshift=-0.25cm, text=blue]{$19$}}] at (0, 18) (a10) {};

\node[dot=4pt, fill=black, label={[xshift=0.55cm, yshift=-0.25cm, text=blue]{$2$}}] at (4, 0) (b1) {};
\node[dot=4pt, fill=black, label={[xshift=0.55cm, yshift=-0.25cm, text=blue]{$4$}}] at (4, 2) (b2) {};
\node[dot=4pt, fill=black, label={[xshift=0.55cm, yshift=-0.25cm, text=blue]{$6$}}] at (4, 4) (b3) {};
\node[dot=4pt, fill=black, label={[xshift=0.55cm, yshift=-0.25cm, text=blue]{$8$}}] at (4, 6) (b4) {};
\node[dot=4pt, fill=black, label={[xshift=0.55cm, yshift=-0.25cm, text=blue]{$10$}}] at (4, 8) (b5) {};
\node[dot=4pt, fill=black, label={[xshift=0.55cm, yshift=-0.25cm, text=blue]{$12$}}] at (4, 10) (b6) {};
\node[dot=4pt, fill=black, label={[xshift=0.55cm, yshift=-0.25cm, text=blue]{$14$}}] at (4, 12) (b7) {};
\node[dot=4pt, fill=black, label={[xshift=0.55cm, yshift=-0.25cm, text=blue]{$16$}}] at (4, 14) (b8) {};
\node[dot=4pt, fill=black, label={[xshift=0.55cm, yshift=-0.25cm, text=blue]{$18$}}] at (4, 16) (b9) {};
\node[dot=4pt, fill=black, label={[xshift=0.55cm, yshift=-0.25cm, text=blue]{$20$}}] at (4, 18) (b10) {};

\draw[thick,black] (a1)-- (a2) -- (a3)-- (a4)-- (a5) -- (a6) -- (a7) -- (a8) -- (a9) -- (a10);
\draw[thick,black] (b1)-- (b2) -- (b3)-- (b4)-- (b5) -- (b6) -- (b7) -- (b8) -- (b9) -- (b10);

\draw[thick, black] (b1)-- (a2);
\draw[thick,black] (b2)-- (a3);
\draw[thick,black] (b4)-- (a5);
\draw[thick,black] (b8)-- (a9);
\draw[thick,black] (b9)-- (a10);

\draw[thick,black] (a3)-- (b4);
\draw[thick,black] (a5)-- (b6);
\draw[thick,black] (a6)-- (b7);
\draw[thick,black] (a7)-- (b8);
\end{tikzpicture}
}
\end{subfigure}
~
\begin{subfigure}{0.40\textwidth}
\centering
   \begin{tikzpicture}[inner sep=0in,outer sep=0in]
\node (n) {\begin{varwidth}{6cm}{
\ytableausetup{boxsize=1.25em}
\begin{ytableau} \none & \none[2] & \none[4] & \none[6] & \none[8] & \none[10] & \none[12] & \none[14] & \none[16] & \none[18] & \none[20]  \\ \none[19] & \none & \none & \none & \none & \none & \none & \none & \none & \none &  \\ \none[17] & \none & \none & \none & \none & \none & \none & \none & \none &  &  \\ \none[15] & \none & \none & \none & \none &  &  &  &  &  &  \\ \none[13] & \none & \none & \none & \none &  &  &  \\ \none[11] & \none & \none & \none & \none &  &  \\ \none[9] & \none & \none & \none & \none &  \\ \none[7] & \none & \none &  &  &  \\ \none[5] & \none & \none &  \\ \none[3] & \none &  &  \\ \none[1] &  &  &  \\ \end{ytableau}

}\end{varwidth}};
\end{tikzpicture}

\end{subfigure}
\caption{Generalized snake poset $P$ for $\mathbf{w} = \varepsilon RRLLLRLRR$ and the skew shape corresponding to $P$. 
The two chains give rise to the rows and columns of the skew shape. 
The cover relations $\nearrow$ and $\nwarrow$ correspond respectively to the outer and inner corners of the skew shape.}
\label{fig:snake_poset_example}
\end{figure}
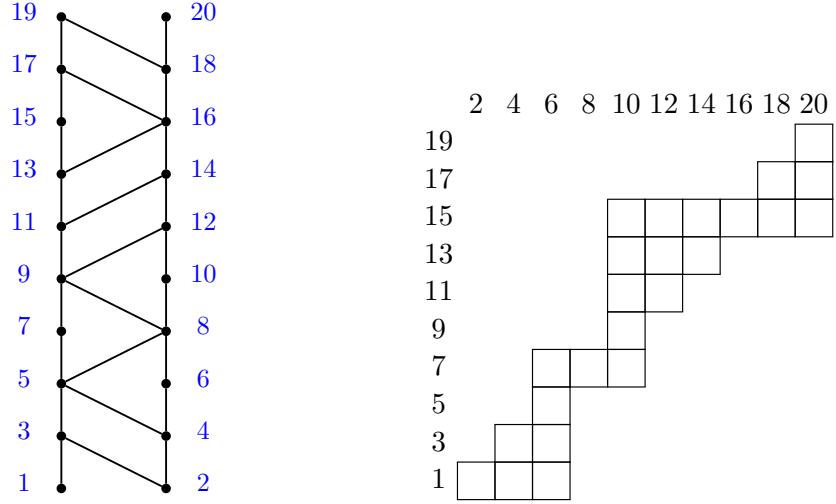

\end{example}

It is straightforward to verify that after adjoining minimum and maximum elements, generalized snake posets are width two distributive lattices.
However, not every width two distributive lattice is a generalized snake poset.
In particular, extending a generalized snake poset by a sequence of maximal and/or minimal elements, and taking ordinal sums of the resulting posets, also yield width two distributive lattices.

\subsection{Real-rooted polynomial lemmas}

Regarding real-rooted polynomials, we follow the notation and conventions of Br\"and\'en's survey article~\cite{Branden2015}. 
Let $f$ and $g$ be polynomials with positive leading coefficients and
real, non-positive
zeros, $a_{i}$ and $b_{i}$, respectively.
We say that $f$ \defin{interlaces} $g$, and we write
$\defin{f \interl g}$ if
\[
\dotsm \leq a_{3}\leq b_{3}\leq a_{2}\leq b_{2}\leq a_{1}\leq b_{1} \leq 0.
\]
Note that this condition implies that $\deg(f)=\deg(g)$ or $\deg(f)+1=\deg(g)$. By convention, any polynomial of degree $0$ interlaces a polynomial of degree $1$, and the identically-zero function interlaces (and is interlaced by) any real-rooted polynomial with positive leading coefficient, i.e., $0 \interl f, f \interl 0, 0 \interl 0$. We say that sequence of real-rooted polynomials $(f_{k})_{k=1}^{n}$ is an \defin{interlacing sequence of polynomials} if $f_{i} \interl f_{j}$ for every $1 \leq i <j \leq n$. 

The usefulness of interlacing arises from the following result, commonly referred to as Obreschkoff's theorem.

\begin{proposition}[{Obreschkoff,~\cite[Satz 5.2]{Obreschkoff1963Interlacing}\label{prop:obreschkoff_thm}}]
Let $f, g \in \mathbb{R}[t]$ be non-zero polynomials with real roots. Then $\alpha f + \beta g$ is real-rooted for all $\alpha, \beta \in \mathbb{R}$ if and only if
$f \interl g $ or $ g \interl f$.
\end{proposition}

The following elementary results about interlacing are popularly invoked as Wagner's lemma and can be seen as refinements of the previous proposition.

\begin{lemma}[{Wagner,~\cite[Section~3]{Wagner1992}}]\label{lem:wagner}
Let $f, g, h \in \setR[t]$ be real-rooted polynomials with only real, non-positive zeros and positive leading coefficients.
Then
\begin{enumerate}                                     \item[(i)] If $h \interl f, g$, then for any $\alpha, \beta \in \mathbb{R}_{\geqslant 0}$, 
  $h \interl \alpha f + \beta g$.  
  \item[(ii)] If $f, g \interl h$, then for any $\alpha, \beta \in \mathbb{R}_{\geqslant 0}$,
   $\alpha f + \beta g \interl h$.                   \item[(iii)] If $f \interl g$, then $g \interl xf$.                           \item[(iv)] If $f \interl g$, then for all $\alpha, \beta \in \mathbb{R}_{\geqslant    
  0}$, $f \interl \alpha f + \beta g \interl g$.      \end{enumerate}
\end{lemma}

We will need the following Wagner-type lemma about interlacing with minus signs. 

\begin{lemma}[{Br\"and\'en and Saud,~\cite[Lemma 2.1 (iv)]{BrandenSaud2025ChainGeometric}}]\label{lem:g_interlaces_f-tg}
    Let $f, g$ be monic polynomials with $\deg(f) = \deg(g)+1$ and all real, non-positive roots. Suppose $h = f - t\cdot g$ has positive leading coefficient. If $g \interl f$, then $h$ is real-rooted and $g \interl h \interl f$.
\end{lemma}

The following theorem will be a key tool in this work.

\begin{theorem}[{Liu and Wang,~\cite[Theorem 2.6]{Liu2007}}]\label{thm:LW-main-interlacing}               
Let $f$, $g$, $F$, $G$ be four  polynomials in $\mathbb{R}[t]$ with positive leading coefficients that satisfy the following conditions.
\begin{enumerate}                               
      \item[(a)] $F(t) = a(t)f(t) + b(t)g(t)$ and $G(t) = c(t)f(t) + d(t)g(t)$, where $a(t)$, $b(t)$, $c(t)$, $d(t)$ are real polynomials such that $\deg F =  
  \deg G$ or $\deg G + 1$.                        
      \item[(b)] $f, g$ have real zeros, and $g \interl f$.                                 \item[(c)] $\Delta(t) := a(t)d(t) - b(t)c(t) \geq 0$ whenever $G(t) = 0$.       
  \end{enumerate}                                                
Suppose that either $c(t)$ is a positive constant and $\deg G \leq \deg g + 1$, or $d(t)$ is a positive constant and $\deg G \leq \deg f$. Then $F, G$ have real zeros and $G \interl F$.                                
\end{theorem}

\section{Narayana, Delannoy, and generalized snake posets}
\label{sec:polynomials}

In this section, we develop technical results that we will require regarding Narayana polynomials and $h^*$-polynomials for order polytopes of generalized snake posets.
We also give a simplified proof of a result of Lee, Vindas-Mel\'endez, and Wang regarding Delannoy polynomials.

Recall that the Narayana numbers and polynomials are defined by 
\[
N(n, k) := \frac{1}{n}\binom{n}{k}\binom{n}{k-1} \quad  \text{for $1 \leq k \leq n$} \quad \text{and} \quad N_{n}(t) := \sum_{k=1}^{n}N(n, k)t^{k}\, ,
\]
with $N(0, 0) :=1$. 
We consider the following minor modification of the Narayana polynomials, so that the constant term is equal to $1$: 
\[
P_{n}(t) := t^{-1} N_{n+1}(t).
\]
As shown by Alexandersson and Jal~\cite[Example 2.4]{AlexanderssonJal2024RookMatroid}, the $P_{n}$ polynomials are the non-nesting rook polynomials corresponding to staircase shapes $\delta_{n}$ (and the conjugate shape $\delta_{n}'$).
These polynomials satisfy the following known polynomial version of log-concavity.

\begin{lemma}\label{lem:NarayanaTuranInequality}
  For each $n \geq 1$ and $t \leq 0$, \[
  P_{n}(t)^{2} - P_{n-1}(t)P_{n+1}(t) \geq 0.
  \]  
\end{lemma}

\begin{proof}
   Let $J_{n}^{(\alpha, \beta)}(t)$ denote the Jacobi polynomials with parameters $\alpha, \beta$ \cite[pg. 254]{Rainville1960SpecialFunctions}. It is a well-known fact that the Narayana polynomials can be obtained as a specialization of the Jacobi polynomials as follows: \[
   P_{n}(t) = \dfrac{1}{n+1}(t-1)^{n}J_{n}^{(1, 1)}\left (\frac{t+1}{t-1} \right).
   \] 
   Renormalize the Jacobi polynomials by setting $R_{n}^{(\alpha, \beta)}(t) = \frac{1}{J_{n}^{(\alpha, \beta)}(1)}J_{n}^{(\alpha, \beta)}(t)$. Then, for $t \leq 0$, set $x = \frac{t+1}{t-1} \in [-1, 1]$ and consider 
\[
 P_{n}(t)^{2} - P_{n-1}(t)P_{n+1}(t) = (t-1)^{2n}(R_{n}^{(\alpha, \beta)}(x)^{2} - R_{n-1}^{(\alpha, \beta)}(x)R_{n+1}^{(\alpha, \beta)}(x)),
\] 
where $(\alpha, \beta) = (1,1)$. By work of Gasper~\cite{Gasper1972TuranInequalityJacobi}, the non-negativity of the bracketed expression in the right-hand side above holds for all $\alpha, \beta > -1$ and $x \in [-1, 1]$.
The result follows. 
\end{proof}

We will be dealing with the modified Narayana polynomials and their differences, which naturally arise in the context of non-nesting rook polynomials of truncated staircases. 
For $i \geq 0$, let $\mu_{n, i} := (n, n-1, \ldots , n-i+1)$ be a \defin{truncated staircase shape.} 
When $i =0$, we treat $\mu_{n, i}$ as the empty partition.
Note that the non-nesting rook polynomial $M_{\mu_{n, i}}(t)$ has degree $i$, and that $M_{\mu_{n, n}}(t)=P_n(t)$.

The following proposition gives a recurrence for the non-nesting rook polynomial when the first column of a board is deleted, or equivalently, when each part is reduced by one.

\begin{proposition}\label{prop:columndeletion}
    Let $\lambda$ be a partition with $n$ parts. We have \begin{equation}\label{eq:rook_rec}
M_{\lambda}(t) = M_{\lambda - \mathbf{1}}(t) +t\sum_{i=0}^{n-1}M_{(\lambda - \mathbf{1})^{(i)}}(t),
\end{equation}

\noindent where $\mathbf{1}$ denotes the all-ones vector and $\mu^{(i)}$ denotes the partition obtained from $\mu$ by truncating to its first $i$ parts when $i \geq 1$, and the empty partition for $i = 0.$
\end{proposition}

\begin{proof}
Let $\NN_{k}(\lambda)$ denote the set of non-nesting rook placements on $\lambda$ of size $k$ and $r_{k}^{\lambda} = |\NN_{k}(\lambda)|$. It suffices to show that \[
r_{k}^{\lambda} = r_{k}^{\lambda - \mathbf{1}} + \sum_{i=0}^{n-1}r_{k-1}^{(\lambda - \mathbf{1})^{(i)}}.
\]

\noindent Let $\rho$ be an element in $\NN_{k}(\lambda)$. By considering the occupancy of the first column of $\rho$, either:

\begin{enumerate}
    \item There is no rook in the first column of $\rho$. Delete the first column of $\rho$ to get an element $\rho' \in \NN_{k}(\lambda - \mathbf{1})$. Alternatively, 
    \item There is a rook in the first column of $\rho$ at row $i$ for some $1 \leq i \leq n$. In this case, delete $\{(m, j) \in \lambda: m \geq i\}$ from $\rho$ to get an element $\rho'' \in \NN_{k-1}((\lambda - \mathbf{1})^{(i-1)})$.
\end{enumerate}
    
\noindent The first case contributes a term of $r_{k}^{\lambda - \mathbf{1}}$, while the second case contributes a term of $\sum_{i=0}^{n-1}r_{k-1}^{(\lambda - \mathbf{1})^{(i)}}$. This completes the proof.
\end{proof}

Define $G_{n}(t) := \sum_{i=0}^{n-1}M_{\mu_{n, i}}(t)$, and note that $G_{n}$ has degree $n-1$. Proposition~\ref{prop:columndeletion} relates $G_{n}$ and $P_{n}$ by \begin{equation}\label{eq:g_relating_p}
t\cdot G_{n-1} = P_{n} - (1+t)P_{n-1}, 
\end{equation}

\noindent which allows for the following simple interlacing to be noted.

\begin{lemma}\label{lem:g_interlaces_p}
    For all $n \geq 1$, $G_{n} \interl P_{n}$. 
\end{lemma}

\begin{proof}
    It is well-known that the $P_{n}$ are real-rooted and $P_{n} \interl P_{n+1}$~\cite{Brenti1989UnimodalAMS, Liu2007}. Since $P_{n+1}-t \cdot P_{n}$ clearly has positive coefficients by (\ref{eq:g_relating_p}), we can use Lemma~\ref{lem:g_interlaces_f-tg} to conclude that \[
    P_{n} \interl P_{n} + t \cdot G_{n} \interl  P_{n+1}.
    \]
    Since $P_{n} \interl P_{n} + t\cdot G_{n}$,  by Proposition~\ref{prop:obreschkoff_thm} we also have $P_{n} \interl t \cdot G_{n}$, which by Lemma~\ref{lem:wagner} implies that $G_{n} \interl P_{n}$.
\end{proof}

We will use the following technical lemma about the interlacing of Narayana polynomials.
\begin{lemma}\label{lem:weighted-narayana-interlacing}
    For all $m \geq 2$, \[
(\lambda t + \nu)P_{m-1} + P_{m} \interl (\lambda t + \nu)P_{m} + P_{m+1} \quad \text{for all $\lambda \geq 0, \nu \geq -1$.}
\]
\end{lemma}

\begin{proof}
   Fix $m \geq 2$ and $\lambda \geq 0, \nu \geq -1$. We will establish this using Theorem~\ref{thm:LW-main-interlacing}. Recall that we have the following three-term recurrence for the Narayana polynomials~\cite{coker2003,sulankeparallel}: 
    \begin{equation}\label{eq:narayana-three-term}
(m+3)P_{m+1}(t) = (2m+3)(1+t)P_{m}(t) - m(1-t)^2P_{m-1}(t), 
\end{equation}
with $P_{0}(t) =1, P_{1}(t) = 1+t$. Using the notation of Theorem~\ref{thm:LW-main-interlacing}, we set 
  \begin{align*}
      f &= P_m, & g &= P_{m-1}, \\
      F &= (\lambda t + \nu) P_m + P_{m+1}, & G &= (\lambda t + \nu) P_{m-1} + P_m, \\
      a(t) &= \frac{(m+3)(\lambda t + \nu) + (2m+3)(1+t)}{m+3}, & b(t) &= -\frac{m(1-t)^2}{m+3}, \\
      c(t) &= 1, & d(t) &= \lambda t + \nu.
  \end{align*}
  By~(\ref{eq:narayana-three-term}), we can write $F(t) = a(t)f(t) +b(t)g(t)$, while $G(t) = c(t)f(t)+d(t)g(t)$ is immediate.
  This verifies condition \textit{(a)} of Theorem~\ref{thm:LW-main-interlacing}. Condition \textit{(b)} requires that $g \interl f$, which holds since $P_m$ the modified Narayana polynomials form an interlacing sequence. 
  Condition \textit{(c)} requires us to verify that $\Delta(t)$ is non-negative at every real root of $G$. We first compute  \[                                  
      \Delta(t) = a(t)\,d(t) - b(t)\,c(t) \;=\; \frac{(m+3)(\lambda t + \nu)^2 + (2m+3)(1+t)(\lambda t + \nu) + m(1-t)^2}{m+3}.
  \] 
  Now $G(t)$ can be written as $G(t) = \lambda tP_{m-1} + (\nu+1)P_{m-1} + (P_{m}-P_{m-1})$, where $P_{m}-P_{m-1}$ has non-negative coefficients from~(\ref{eq:g_relating_p}). Since $\lambda \geq 0$ and $ \nu \geq -1$, $G(t)$ also has non-negative coefficients and hence any real root of $G$ must be non-positive. Let $r \leq 0$ be some root of $G$ so that $(\lambda r + \nu)P_{m-1}(r) = -P_{m}(r)$. Substituting this in $(m+3)\Delta(r)P_{m-1}(r)^{2}$, we get the following chain of equalities: \begin{align*}
      (m+3)\Delta(r)P_{m-1}(r)^{2}&= (m+3)P_{m}(r)^{2} - (2m+3)(1+r)P_{m}(r)P_{m-1}(r) + m(1-r)^{2}P_{m-1} (r)^{2} \\ & = (m+3)P_{m}(r)^{2} - P_{m-1}(r)((2m+3)(1+r)P_{m}(r) - m(1-r)^{2}P_{m-1}(r)) \\&=
      (m+3)[P_{m}(r)^{2} - P_{m-1}(r)P_{m+1}(r)].
      \end{align*}

In the first equality, we performed the aforementioned substitution. In the second equality, we rearranged terms. In the third equality, we used~(\ref{eq:narayana-three-term}) again. 

 Since consecutive $P_{n}$ have no zeros in common and $(\lambda r + \nu)P_{m-1}(r) = -P_{m}(r)$, we have that $P_{m-1}(r) \neq 0$. Thus, from the displayed equation above, it follows that $\Delta(r) \geq 0$ if and only if $P_{m}(r)^{2} - P_{m-1}(r)P_{m+1}(r) \geq 0$ for $r$ a root of $G$. 
 Since the latter inequality holds for all non-positive $r$ by Lemma~\ref{lem:NarayanaTuranInequality}, it follows that condition \textit{(c)} is also verified. Since $c(t)$ is a positive constant and $\deg G \leq \deg g +1$, we may apply Theorem~\ref{thm:LW-main-interlacing} to get that $G, F$ are real-rooted and satisfy $G \interl F$, as desired.  
\end{proof}

We now turn our attention to connecting Narayana polynomials with generalized snake posets. An equivalent version of the following recurrence was derived by Lee, Vindas-Meléndez, and Wang~\cite[Theorem 4.14]{lee2024generalizedsnakeposetsorder}. 
However, in the original form, the terms in the right-hand side were not manifestly positive, and the use of non-nesting rook polynomials simplifies the proof considerably. We take this opportunity to give a shorter proof and use it to derive one other fact. Below, given a generalized snake word $\mathbf{w} = \varepsilon w_{1}\cdots w_{n}$, we use $\mathbf{w}[:k]$ to mean $\varepsilon w_{1}\cdots w_{k}$. 

\begin{theorem}
    \label{thm:snakerecurrences}
    Suppose $\mathbf{w}=\varepsilon w_1\cdots w_n$ is a non-constant word defining a generalized snake poset $P(\mathbf{w})$. 
    Set $k:= \max\{j \in [n-1]:w_j\neq w_n\}$. Then
    \begin{equation}\label{eq:recurrence2}
    M_{\mathbf{w}}(t) = M_{\mathbf{w}[:k]}(t)P_{n-k}(t) + tM_{\mathbf{w}[:k-1]}(t)G_{n-k}(t)\, .
    \end{equation}
\end{theorem}

\begin{proof}
   Let $D$ be the skew shape corresponding to $\mathbf{w}$ and suppose its rows are labeled $1,2,\ldots $ from the bottom and its columns are labeled $1,2,\ldots $ from left to right; this is obtained by applying the transformation $i\to \lceil i/2 \rceil$ to the row/column labeling used previously, e.g., in Figure~\ref{fig:snake_poset_example}.
   Let $\rho$ be a rook placement on $D$. 
   Without loss of generality, assume that the final $n-k$ letters of $\mathbf{w}$ are equal to $R$. 
   This means that the rotated staircase $\delta_{n-k+1}^{\mathrm{R}}$ occupies the first $n-k+1$ rows of $D$ from the bottom. 
   Consider the first $n-k$ columns in the row $n-k+1$ of $\delta_{n-k+1}^{\mathrm{R}}$. 

   Suppose there is no rook of $\rho$ in $(n-k+1,\ell)$ for $1 \leq \ell \leq n-k$. Then $\rho$ consists of a pair of non-nesting rook placements, one on the smaller rotated staircase $\delta_{n-k}^{\mathrm{R}}$, and one on the skew shape corresponding to the word $\mathbf{w}[:k]$. 
   Since these shapes are disjoint, this accounts for the first term.

   Suppose there is a rook of $\rho$ in $(n-k+1,\ell)$ for some $1 \leq \ell \leq n-k$. 
   Excluding that single rook, $\rho$ then consists of the pair of rook placements $(\rho', \rho'')$, where $\rho'$ is a rook placement on the shape corresponding to the word $\mathbf{w}[:k-1]$ and $\rho''$ is a rook placement on the skew shape strictly South-West of $(n-k+1,\ell)$. 
   This latter shape is simply the truncated staircase $\mu_{n-k, \ell -1}$ rotated 180 degrees. 
   Summing over all $\ell =1, \ldots , n-k$, and using the definition of $G_{n-k}$, this accounts for the second term. 
\end{proof}

Recall that the \defin{Delannoy numbers} $D(a, b)$ count the number of lattice paths from $(0,0)$ to $(b, a)$ that take East, North and North-East steps. They are defined by the formula:\[
    D(a, b) := \sum_{k = 0}^{\min(a, b)}2^{k}\binom{a}{k}\binom{b}{k}.
    \]
Let $d_{n}(t) := \sum_{i=0}^{n}D(n-i, i)t^{i}$ be the \defin{Delannoy polynomials}. They satisfy the following well-known recurrence:\begin{equation}\label{eq:delannoy-recurrence}
d_{n}(t) = (t+1)d_{n-1}(t)+td_{n-2}(t), \quad  \text{ with } \quad d_{0}(t) := 1 \text{ and } d_{1}(t) := t+1 \, .
\end{equation}

We conclude this section by applying Theorem~\ref{thm:snakerecurrences} to give a simple and direct proof of the fact that the non-nesting rook polynomial corresponding to the regular snake word is a Delannoy polynomial. The following result was originally obtained by Lee, Vindas-Meléndez, and Wang~\cite[Theorem 4.8]{lee2024generalizedsnakeposetsorder}. 
\begin{corollary}\label{cor:reg-snake-is-delannoy}
    Let $S_{n} = \varepsilon LR\ldots LRL$ be the regular snake word. Then $M_{S_{n}}(t) = d_{n+1}(t)$. 
\end{corollary}

\begin{proof}
    It suffices to show that the Delannoy recurrence (\ref{eq:delannoy-recurrence}) holds for $M_{S_{n}}$ with the base cases. When $n = 0$, the skew shape is a single box, and $M_{S_{0}}(t) = t+1 = d_{1}(t)$; when $n =1$, the corresponding skew shape is $(2,1)/(1)$ which has non-nesting rook polynomial equal to $M_{S_{1}}(t) = t^2 + 3t +1 = d_{2}(t)$. Now suppose $n \geq 2$. With the notation of Theorem~\ref{thm:snakerecurrences}, $k = n-1$, and thus $P_{n-k}(t) = t+1$ and $G_{n-k}(t) = 1$. Hence~(\ref{eq:recurrence2}) reduces to~(\ref{eq:delannoy-recurrence}) and we are done.
\end{proof}

\section{$h^*$-real-rootedness}\label{sec:realrootedness}

In this section we present our main result, a proof of Conjecture~5.1~(1) from Lee, Vindas-Meléndez, and Wang~\cite{lee2024generalizedsnakeposetsorder}.  

\begin{theorem}\label{thm:interlacing_of_Ms}
    Let $\mathbf{w}$ be a generalized snake word  of length $n \geq 1$ and let $\mathbf{w}'$ be obtained from $\mathbf{w}$ by deleting $w_{n}$. 
    Then $M_{\varepsilon \mathbf{w}}$ is real-rooted and $M_{\varepsilon \mathbf{w'}}$ interlaces $M_{\varepsilon \mathbf{w}}$, i.e., we have $M_{\varepsilon \mathbf{w}'} \interl M_{\varepsilon \mathbf{w}}$.
\end{theorem}

\begin{proof}
    We proceed by strong induction on $n$. 
    The base case $n=1$ holds immediately, since $M_{\varepsilon R}(t) = M_{\varepsilon L} = t^{2} + 3t +1$ which has real roots $\frac{-3-\sqrt{5}}{2}$ and $\frac{-3+\sqrt{5}}{2}$ that lie on either side of $-1$, the zero of $M_{\varepsilon}(t) = 1+t$. 
    Now suppose $n \geq 2$ and the result holds for all generalized snake words of length less than $n$. 
    If $\mathbf{w}$ equals $\varepsilon LL\ldots L$ or $\varepsilon RR\ldots R$, then the result holds from the real-rootedness and interlacing of Narayana polynomials. 
    The index $k$ from Theorem~\ref{thm:snakerecurrences} is thus well-defined. By Lemma~\ref{lem:wagner}, the conclusion is immediate if $k=n-1$, since $P_{1}(t) = t+1$ and $G_{1}(t) = 1$. 
    So, we may assume $k < n-1$. Define $Q_{n} = P_{n} - P_{n-1}$ and $H_{n} = G_{n} - G_{n-1}$. By applying  (\ref{eq:recurrence2}) to both $\mathbf{w}$ and $\mathbf{w'}$, and using the notation $\mathbf{w}[:k]:=w_1\cdots w_k$, we can write \[
    \begin{pmatrix}
        M_{\mathbf{\varepsilon w'}} \\ M_{\mathbf{\varepsilon w}} - M_{\mathbf{\varepsilon w'}}  
    \end{pmatrix}
    = \begin{pmatrix}
        P_{n-k-1} & G_{n-k-1} \\ Q_{n-k} & H_{n-k}
    \end{pmatrix}
    \begin{pmatrix}
        M_{\varepsilon \mathbf{w}[:k]} \\ 
        t\cdot M_{\varepsilon \mathbf{w}[:k-1]}
    \end{pmatrix}
    \] 

By the inductive hypothesis and Lemma~\ref{lem:wagner}, we have $M_{\varepsilon \mathbf{w}[:k]} \interl t \cdot M_{\varepsilon \mathbf{w}[:k-1]}$. 
We now claim that the matrix in the equation above satisfies \begin{equation}\label{eq:matrix_interlacing}
    (\lambda t + \mu )G_{n-k-1} + H_{n-k} \interl (\lambda t + \mu )P_{n-k-1} + Q_{n-k} \quad \text{for all $\lambda, \mu \geq 0$.}
\end{equation}

Assume that the claim holds. Then the conditions of a theorem of Br\"andén~\cite[Theorem 8.5]{Branden2015} are satisfied and hence the matrix maps interlacing sequences to interlacing sequences. 
In particular, we have $M_{\varepsilon \mathbf{w'}} \interl  M_{\mathbf{\varepsilon w}} - M_{\mathbf{\varepsilon w'}}$, and since $M_{\mathbf{\varepsilon w}} - M_{\mathbf{\varepsilon w'}}$ has non-negative coefficients, all roots are non-positive. 
By applying part \textit{(iv)} of Lemma~\ref{lem:wagner}, it follows that $M_{\varepsilon \mathbf{w}}$ is real-rooted and $M_{\varepsilon\mathbf{w}'} \interl M_{\varepsilon \mathbf{w}}$, completing the induction step. 

We now prove the claim, which will complete the proof.
After reindexing and using the definitions of $Q_{n}$ and $H_{n}$, it is enough to show that for all $m \geq 2$, \begin{equation}\label{eq:last-interlacing}
(\lambda t + \nu)G_{m-1} + G_{m} \interl (\lambda t + \nu)P_{m-1} + P_{m} \quad \text{for all $\lambda \geq 0, \nu \geq -1$.}
\end{equation}
By Lemma~\ref{lem:weighted-narayana-interlacing}, for all $m \geq 2$, \[
(\lambda t + \nu)P_{m-1} + P_{m} \interl (\lambda t + \nu)P_{m} + P_{m+1} \quad \text{for all $\lambda \geq 0, \nu \geq -1$.}
\]
If we set $U_{m} \coloneq (\lambda t + \nu)P_{m-1} + P_{m}$ and $V_{m} \coloneq (\lambda t + \nu)G_{m-1} + G_{m}$, then Lemma~\ref{lem:weighted-narayana-interlacing} says that $U_{m} \interl U_{m+1}$, while we want to show $V_{m} \interl U_{m}$. 
By \eqref{eq:g_relating_p}, we can write $U_{m+1} = (1+t)U_{m} + tV_{m}$. Since $U_{m} \interl U_{m+1}$, by Lemma~\ref{lem:g_interlaces_f-tg}, we also have $U_{m} \interl U_{m} + tV_{m}$. 
It follows that $U_{m} \interl t V_{m}$ which gives us $V_{m} \interl U_{m}$, as desired.
\end{proof}

\section{Tool and computational resource disclosure}\label{sec:aidisclosure}

The authors used Claude Opus 4.7 to discuss the proof strategy of Theorem~\ref{thm:interlacing_of_Ms} and ChatGPT 5.5 to locate the reference in the proof of Lemma~\ref{lem:NarayanaTuranInequality}. All information that was produced required editing and revision by the authors, and the final versions of the resulting proofs were written and checked for correctness by the authors.
All writing for this article was done by the authors.

\bibliographystyle{alpha}
\bibliography{bibliography}
\end{document}